\documentclass[12pt, oneside]{amsart}

\usepackage{amsmath,amssymb,cite,mathrsfs,tikz-cd}
\usepackage[all]{xy}
\usepackage{typearea} 
\usepackage[backref=page]{hyperref} 

\renewcommand*{\backrefalt}[4]{{\tiny%
    (\ifcase #1 Not cited.%
          \or Cited on p.~#2.%
          \else Cited on pp. #2.%
    \fi%
    )}}
    
\usepackage{color} 

\usepackage{epstopdf} 
\usepackage{booktabs}

\newtheorem{teo}{Theorem}[section]
\newtheorem{thm}[teo]{Theorem}
\newtheorem{prop}[teo]{Proposition}
\newtheorem{lemma}[teo]{Lemma}
\newtheorem{cor}[teo]{Corollary}

\newtheorem{defn}[teo]{Definition}

\newtheorem{def-prop}[teo]{Definition-Proposition}

\newtheoremstyle{named}{}{}{\itshape}{}{\bfseries}{.}{.5em}{\thmnote{#3 }#1}
\theoremstyle{named}

\makeatletter
\newcommand{\neutralize}[1]{\expandafter\let\csname c@#1\endcsname\count@}
\makeatother

\numberwithin{equation}{section}

  \newcommand{\A}{\mathbb{A}}
  \newcommand{\C}{\mathbb{C}}

  \newcommand{\Q}{\mathbb{Q}}
  \newcommand{\R}{\mathbb{R}}

  \newcommand{\Z}{\mathbb{Z}}

\newcommand{\PP}{\mathbb{P}}

\newcommand{\calA}{\mathcal{A}}

\newcommand{\calL}{\mathcal{L}}

  \renewcommand{\bar}{\overline}

  \providecommand{\frac}[1]{\operatorname{Frac}(#1)}

  \newcommand{\tor}{\operatorname{tor}}

  \newcommand{\rank}{\operatorname{rank}}

\DeclareMathOperator{\an}{an}

\DeclareMathOperator{\Hom}{Hom}

\DeclareMathOperator{\Lie}{Lie}

\DeclareMathOperator{\Pic}{Pic}

\DeclareMathOperator{\vol}{vol}

  \renewcommand{\lim}{\operatorname{lim}}
  \renewcommand{\deg}{\operatorname{deg}}

\newcommand{\cE}{\mathcal{E}}
\newcommand{\cF}{\mathcal{F}}

\newcommand{\cI}{\mathcal{I}}

\newcommand{\cM}{\mathcal{M}}
\newcommand{\cO}{\mathcal{O}}

\newcommand{\cZ}{\mathcal{Z}}

\usepackage{lmodern}
\usepackage{microtype}
\usepackage{amsmath,amssymb,amsthm,mathtools,mathrsfs}
\usepackage{enumitem}
\usepackage{booktabs}
\usepackage{hyperref}

\newcommand{\Gm}{\mathbb G_{\mathrm m}}

\newcommand{\NS}{\operatorname{NS}}
\newcommand{\Bad}{\operatorname{Bad}}
\newcommand{\gon}{\operatorname{gon}}
\newcommand{\Sym}{\operatorname{Sym}}

\newcommand{\Bfun}{\mathscr B}

\title{On the Torsion Subgroup of the N\'eron--Severi Group}
\author{Kaiyuan Gu}
\date{August 30, 2026}
\hypersetup{
  pdftitle={On the Torsion Subgroup of Neron--Severi Group},
  pdfauthor={Kaiyuan Gu},
  pdfsubject={Proof-audited revised version}
}

\begin{document}

\begin{abstract}
\begingroup
We prove an explicit bound for the torsion subgroup of the N\'eron--Severi group of a smooth projective complex variety. The proof combines the Diophantine-approximation method of \cite{GaoGu_UBC} with Lefschetz fibrations. As an application, we deduce an explicit bound for the torsion sections of a traceless abelian scheme over a complex curve.
\endgroup
\end{abstract}

\maketitle

\tableofcontents

\section{Introduction}

\begingroup
The goal of this paper is to prove an explicit bound for the torsion subgroup of the N\'eron--Severi group of a smooth projective complex variety equipped with a very ample line bundle. This work is a sequel to \cite{GaoGu_UBC}, where we developed a Diophantine-approximation method for uniformity questions concerning small points on abelian varieties.

For later use, define
\[
\Bfun(0):=1,
\qquad
\Bfun(q):=\bigl(768q^5(q+2)\bigr)^q\quad(q\geq1).
\]

Our main result is as follows.
\begin{thm}\label{thm:polarized-variety}
Let $X$ be a smooth projective connected variety over $\C$, endowed with a very ample line bundle $L$. If $n:=\dim X\geq2$, then
\[
\#\NS(X)_{\tor}\leq \Bfun\bigl(q(X,L)\bigr),
\]
where $q(X,L)$ is the genus of a smooth complete-intersection curve cut out by $n-1$ general sections of $L$, namely
\[
q(X,L)=1+\frac12\left(K_X\cdot L^{n-1}+(n-1)L^n\right).
\]
\end{thm}

As an application, we obtain a bound for the torsion sections of an abelian scheme over a curve.
\begin{cor}\label{cor:AV}
Let $\calA/B$ be a traceless abelian scheme over a smooth projective connected complex curve $B$. Put $q=g(B)$ and $g=\dim(\calA/B)$, and define
\[
Q:=\frac{(8g+3)!\,3^{8g}}2(q+1).
\]
Then
\[
\#\calA(B)_{\tor}\leq\Bfun(Q).
\]
\end{cor}

Compared with \cite[Theorem~1.3]{GaoGu_UBC}, this substantially improves the dependence on the genus of the base curve: for fixed relative dimension $g$, the bound in \cite{GaoGu_UBC} is double exponential in $q$, whereas the bound above has the form $\exp(O(q\log q))$.

\subsection{Comparison with the previous result}
In \cite{NeronSeveri}, Hyuk Jun Kweon proved, to the author's knowledge, the first explicit bound for $\NS(X)_{\tor}$.

\begin{thm}[Kweon]\label{thm:Kweon}
Let $\iota:X\hookrightarrow\PP^r$ be a smooth $n$-dimensional projective variety whose homogeneous ideal is generated by polynomials of degree at most $d$. Then
\[
\#\NS(X)_{\tor}\leq2^{d^{r^2+2r\log_2r}}.
\]
\end{thm}

This bound depends on the dimension of the ambient projective space. In the setting of Theorem~\ref{thm:Kweon}, take $L=\iota^*\cO(1)$. The conormal exact sequence gives the crude estimate
\[
K_X\cdot L^{n-1}\leq\bigl((r-n)d-r-1\bigr)L^n,
\]
and hence
\[
q(X,L)\leq1+\frac12(r-n)(d-1)L^n.
\]
B\'ezout's theorem gives $L^n\leq d^{r-n}$, so Theorem~\ref{thm:polarized-variety} yields
\[
\#\NS(X)_{\tor}
\leq
\Bfun\left(1+\frac12(r-n)(d-1)d^{r-n}\right).
\]
In particular, this estimate does not involve a double exponential in $r$.

\subsection{Strategy of proof}
The proof has two parts.

\subsubsection{The Gao--Gu framework}
Following \cite{GaoGu_UBC}, and using a rank-one improvement of the local bump functions, we prove the following key result.
\begin{thm}[=Theorem~\ref{thm:P1_Proof}]\label{thm:P1}
Let $K=\C(t)$, and let $A/K$ be a semistable abelian variety of dimension $g$, endowed with a polarization $L$. Assume that the toric rank of $A$ is at most one at every place. Let $A_0\subset A$ be the maximal abelian subvariety having everywhere good reduction. If $\Gamma\subset A(K)_{\tor}$ is finite and $\Gamma\cap A_0(K)=\{0\}$, then
\[
\#\Gamma\leq h^0(A,L)\bigl(768g^5(g+2)\bigr)^g.
\]
\end{thm}

\subsubsection{A Lefschetz fibration}
To bound $\NS(X)_{\tor}$ for a general polarized variety $(X,L)$, we first cut $X$ by general members of $|L|$. Integral weak Lefschetz reduces the problem to a smooth surface, which we still denote by $X$. We then choose a Lefschetz pencil in $|L|$, blow up its base locus, and apply Theorem~\ref{thm:P1} to the Jacobian of the generic fiber.

\subsection{Acknowledgements}
I thank Ziyang Gao and Xinyi Yuan for helpful discussions, and especially Xinyi Yuan for pointing out the usefulness of Lefschetz fibrations. KG is partially supported by Beijing Natural Science Foundation 24QY0003.
\endgroup

\subsection{Disclosure of AI use}

After Xinyi Yuan pointed out the method of making a Lefschetz fibration, the author realized that a bound on the torsion subgroup of N\'eron--Severi group can be proved in this way. He then used AI to find an optimal bump function $\varphi_*$ as in Lemma~\ref{lem:sharp-bump}. This yields a slightly better bound in the toric rank 1 case than the previous one obtained by \cite{GaoGu_UBC}. After the author finished the draft, he used AI to audit the proof and check the grammar.

\section{Notations and preliminaries}

\begingroup
Let $B$ be a smooth projective connected complex curve and let
\[
K=\C(B),\qquad K_0=\C(\PP^1),\qquad D:=\gon(B)=[K:K_0],
\]
where a morphism $B\to\PP^1$ of minimal degree has been fixed. The set $M_K$ of places of $K$ is naturally identified with $B(\C)$. For $v\in M_K$, let $K_v$ be the completion, $\cO_v$ its valuation ring, and $\varpi_v$ a uniformizer. We normalize the absolute value by $|\varpi_v|_v=e^{-1}$.

\subsection{Abelian varieties over function fields}
Let $A/K$ be an abelian variety of dimension $g\geq1$, endowed with a symmetric ample line bundle $L$, and assume that $A$ has semistable reduction. Let $\mathscr A^\circ$ be the identity component of its N\'eron model over $B$, let $e:B\to\mathscr A^\circ$ be the zero section, and put
\[
\omega_A:=\det e^*\Omega^1_{\mathscr A^\circ/B},
\qquad
h_{\mathrm{Fal}}(A):=\frac1D\deg\omega_A.
\]
For $v\in M_K$, let $r_v$ be the toric rank of the special fiber and let $\Phi_v$ be the component group of the N\'eron model. Since the residue field is $\C$, we write
\[
\#\Phi_v:=\#\Phi_v(\C).
\]
Throughout Sections~3--6 we assume that $r_v\leq1$ at every place. Thus $r_v=0$ at a place of good reduction, while $r_v=1$ at every place in
\[
\Bad(A/K):=\{v\in M_K:A\text{ has bad reduction at }v\}.
\]
We use the two normalized boundary quantities
\[
S(A):=\frac1D\sum_{v\in\Bad(A/K)}\frac1{\#\Phi_v},
\qquad
h_\partial(A):=\frac1D\sum_{v\in\Bad(A/K)}\#\Phi_v.
\]

\subsection{Adelically normed vector spaces}
Let $\cF$ be a vector bundle on $B$, with generic fiber $F=\cF\otimes_{\cO_B}K$. Its model norm on $F_v:=F\otimes_KK_v$ is the gauge norm of the lattice $\cF\otimes_{\cO_B}\cO_v$:
\[
\|f\|_v:=\inf_{a\in K_v}\bigl\{|a|_v:f\in a(\cF\otimes_{\cO_B}\cO_v)\bigr\}.
\]

\begin{defn}
An \emph{adelically normed vector space} over $K$ is a finite-dimensional $K$-vector space $E$ together with non-archimedean norms $\|\cdot\|_v$ on $E_v:=E\otimes_KK_v$, one for every $v\in M_K$, such that the norms are induced by a single vector bundle on some nonempty Zariski open subset of $B$.
\end{defn}

For a nonzero $f\in E$, its normalized height is
\[
h_{\bar E}(f):=-\frac1D\sum_{v\in M_K}\log\|f\|_v.
\]
The sum is finite by coherence. If $0\neq F\subseteq E$ is a $K$-subspace, the restricted norms make $F$ an adelic subspace. For a nonzero $u\in\det F$, define
\[
\widehat{\deg}_{\mathrm n}(\bar F)
:=-\frac1D\sum_{v\in M_K}\log\|u\|_{\det F,v},
\qquad
\widehat\mu(\bar F):=\frac{\widehat{\deg}_{\mathrm n}(\bar F)}{\dim F}.
\]
These quantities do not depend on $u$, by the product formula. We write
\[
\widehat\mu_{\max}(\bar E)
:=\max_{0\neq F\subseteq E}\widehat\mu(\bar F).
\]
The effective sections are
\[
\widehat H^0(\bar E):=\{f\in E:\|f\|_v\leq1\text{ for every }v\}.
\]
For a model norm induced by a vector bundle $\cE$ on $B$, this is $H^0(B,\cE)$.

\subsection{Adelic line bundles and small sections}
Let $Z/K$ be a projective variety and let $\bar M=(M,(\|\cdot\|_{M,v})_v)$ be a continuously adelically metrized line bundle. For $s\in H^0(Z,M)$ put
\[
\|s\|_{\bar M,\sup,v}:=\sup_{z\in Z_v^{\an}}\|s(z)\|_{M,v},
\qquad
h_{\bar M}(s):=-\frac1D\sum_{v\in M_K}\log\|s\|_{\bar M,\sup,v}.
\]
An adelic metric is called \emph{semipositive} (or nef) if it is a uniform limit of metrics induced by relatively nef models. We use the adelic intersection products and arithmetic volume in the conventions of \cite[Section~2.5]{GaoGu_UBC}. If $d=\dim Z$, then
\[
\widehat\vol(\bar M)
:=\limsup_{m\to\infty}\frac{(d+1)!}{m^{d+1}}
\dim_\C\widehat H^0(Z,m\bar M).
\]
For nef $\bar M$, the arithmetic Hilbert--Samuel theorem gives
\begin{equation}\label{eq:arithmetic-HS}
\widehat\vol(\bar M)=\bar M^{d+1};
\end{equation}
see \cite[Theorem~2.4]{GaoGu_UBC}. Combining this with the function-field Minkowski theorem \cite[Theorem~2.3]{GaoGu_UBC} gives the following consequence used below: if $M$ is ample and $\bar M$ is nef, then there are integers $m_j\to\infty$ and nonzero sections $s_j\in H^0(Z,m_jM)$ such that
\begin{equation}\label{eq:small-section-input}
\frac1{m_j}h_{m_j\bar M}(s_j)
\geq
\frac{\bar M^{d+1}}{D(d+1)M^d}+o(1)
\qquad(j\to\infty).
\end{equation}

For the symmetric ample line bundle $L$ on $A$, we always use its canonical semipositive adelic metric and write the resulting adelic line bundle as $\bar L$. It is characterized by the cubical relation
\[
[n]^*\bar L\simeq\bar L^{\otimes n^2}\qquad(n\geq2),
\]
and its associated height on $A(\bar K)$ is the N\'eron--Tate height. In particular, $h_{\bar L}(x)=0$ for every torsion point $x$.

If $\mathbf f=(f_v)_v$ is a coherent family of continuous local functions, with $f_v=0$ for all but finitely many $v$, we write $\bar L(\epsilon\mathbf f)$ for the metric obtained from $\bar L$ by
\[
\|\cdot\|_{L(\epsilon\mathbf f),v}
=e^{-\epsilon f_v}\|\cdot\|_{L,v}.
\]
\endgroup

\section{Bump function on a one-dimensional skeleton}

\begingroup
The rank-one improvement begins with a small variational problem. We use periodic coordinates $x\in\R/\Z$.

\begin{lemma}\label{lem:sharp-bump}
For every $\eta>0$ there is a smooth, nonnegative, one-periodic function $\psi_\eta:\R\to\R$ such that
\begin{equation}\label{eq:bump-properties}
\psi_\eta(0)=0,
\qquad
|\psi_\eta''|\leq1,
\qquad
\int_0^1\psi_\eta(x)\,dx\geq\frac1{32}-\eta.
\end{equation}
\end{lemma}

\begin{proof}
Consider the one-periodic $C^{1,1}$ function whose restriction to $[0,1]$ is
\begin{equation}\label{eq:bang-bang}
\psi_*(x)=
\begin{cases}
\frac12x^2,&0\leq x\leq\frac14,\\[2mm]
\frac{-8x^2+8x-1}{16},&\frac14\leq x\leq\frac34,\\[2mm]
\frac12(1-x)^2,&\frac34\leq x\leq1.
\end{cases}
\end{equation}
It is nonnegative, $\psi_*(0)=0$, and its second derivative is $+1,-1,+1$ on the three indicated intervals. Direct integration gives
\[
\int_0^1\psi_*(x)\,dx=\frac1{32}.
\]
Let $\widetilde\psi_\delta$ be its periodic convolution with a sufficiently narrow nonnegative even mollifier. Then $\widetilde\psi_\delta$ is smooth and one-periodic, and $|\widetilde\psi_\delta''|\leq1$. Choose a minimizer $x_\delta$ and set
\[
\psi_\eta(x):=\widetilde\psi_\delta(x+x_\delta)-\min\widetilde\psi_\delta.
\]
Thus $\psi_\eta$ is smooth, nonnegative, one-periodic, and satisfies $\psi_\eta(0)=0$ and $|\psi_\eta''|\leq1$. Since $\widetilde\psi_\delta\to\psi_*$ uniformly and $\min\widetilde\psi_\delta\to0$, choosing $\delta$ sufficiently small gives the required lower bound for the integral.
\end{proof}

We now place this function on the rank-one skeleton at a bad place $v$. Raynaud uniformization gives
\[
\Sigma_v=X^*_{v,\R}/Y_v.
\]
Choose a generator of the rank-one lattice $X_v^*$, thereby identifying $X_v^*$ with $\Z$ and $Y_v$ with $n_v\Z$, where
\[
n_v=[X_v^*:Y_v]=\#\Phi_v.
\]
Let $b_v>0$ be the coefficient of the polarization quadratic form in this coordinate. The determinant identity for the polarized Raynaud extension gives
\begin{equation}\label{eq:bv}
b_v=\frac{d_v}{\#\Phi_v},
\qquad
d_v=[X_v:\lambda_{T,L}^*(Y_v)]\geq1;
\end{equation}
in particular, $b_v\geq1/\#\Phi_v$. We use the tropicalization, curvature, component-group, and canonical-measure descriptions from \cite[Sections~4.2--5.1]{GaoGu_UBC}.

\begin{prop}\label{prop:local}
For every $\eta>0$ there is a smooth function $f_{v,\eta}$ on $\Sigma_v$ such that:
\begin{enumerate}[label=(\roman*)]
\item $f_{v,\eta}$ vanishes on $\Phi_v=X_v^*/Y_v$ and is invariant under translation by $\Phi_v$;
\item multiplying the canonical metric of $L$ by $e^{-a f_{v,\eta}\circ\operatorname{trop}_v}$ gives a relatively semipositive metric for every $a\in[-1,1]$;
\item if $\alpha_v=c_1(L,\|\cdot\|_{L,v})$ and $\mu_v=\deg_L(A)^{-1}\alpha_v^g$, then
\begin{equation}\label{eq:local-mean}
I_{v,\eta}:=\int_{A_v^{\an}}
(f_{v,\eta}\circ\operatorname{trop}_v)\,d\mu_v
\geq\frac{1/32-\eta}{\#\Phi_v};
\end{equation}
\item with $\beta_{v,\eta}=d'd''(f_{v,\eta}\circ\operatorname{trop}_v)$, one has $\beta_{v,\eta}^2=0$.
\end{enumerate}
\end{prop}

\begin{proof}
In the coordinate above, put
\[
F_{v,\eta}(x):=b_v\psi_\eta(x)\qquad(x\in\R).
\]
The function is one-periodic, hence also $n_v$-periodic, so it descends to $\Sigma_v=\R/n_v\Z$. It vanishes at every integral point and is invariant under integral translation; this proves (i).

The tropical part of the canonical first Chern form in this rank-one direction is
$b_v\,d'x\wedge d''x$, 
whereas
\[
d'd''F_{v,\eta}=b_v\psi_\eta''(x)d'x\wedge d''x.
\]
Thus $|\psi_\eta''|\leq1$ implies (ii). The canonical measure pushes forward under $\operatorname{trop}_v$ to Haar probability measure on $\Sigma_v$. Hence, using one-periodicity,
\[
I_{v,\eta}
=\frac1{n_v}\int_0^{n_v}b_v\psi_\eta(x)\,dx
=b_v\int_0^1\psi_\eta(x)\,dx
\geq\frac{1/32-\eta}{\#\Phi_v},
\]
where the last inequality follows from \eqref{eq:bv}. Finally, $\beta_{v,\eta}$ is a $(1,1)$-form supported in a single tropical direction, so its square is zero. This proves (iv).
\end{proof}
\endgroup

\section{Existence of small sections}

\begingroup
Let $\bar L$ be the canonical adelic line bundle introduced in Section~2. The relation $[n]^*\bar L\simeq\bar L^{\otimes n^2}$ and the projection formula give
\[
n^{2g}\bar L^{g+1}=([n]^*\bar L)^{g+1}=n^{2g+2}\bar L^{g+1}
\qquad(n\geq2),
\]
so
\begin{equation}\label{eq:canonical-self-intersection}
\bar L^{g+1}=0.
\end{equation}

Fix $\eta>0$, abbreviate $f_v=f_{v,\eta}$ and $I_v=I_{v,\eta}$ for $v\in\Bad(A/K)$, and put $f_v=0$ at the places of good reduction. Let $\mathbf f=(f_v)_{v\in M_K}$. For $0<\epsilon\leq1$, let $\bar L(\epsilon\mathbf f)$ be the adelic line bundle whose norm at a bad place is
\[
\|\cdot\|_{L(\epsilon\mathbf f),v}
=e^{-\epsilon f_v\circ\operatorname{trop}_v}\|\cdot\|_{L,v}.
\]
By Proposition~\ref{prop:local}(ii), this adelic metric is semipositive. Write $\overline\cO(\epsilon\mathbf f)$ for the corresponding metrized trivial line bundle. Using \eqref{eq:canonical-self-intersection}, multilinearity of adelic intersections gives
\begin{equation}\label{eq:intersection-expansion}
\bar L(\epsilon\mathbf f)^{g+1}
=
\sum_{k=0}^{g}\binom{g+1}{k+1}
\bar L^{g-k}\cdot\overline\cO(\epsilon\mathbf f)^{k+1}.
\end{equation}

Put
\[
\alpha_v=c_1(L,\|\cdot\|_{L,v}),
\qquad
\beta_v=d'd''(f_v\circ\operatorname{trop}_v).
\]
By Proposition~\ref{prop:local}(iv), $\beta_v^2=0$. Consequently every term in \eqref{eq:intersection-expansion} with $k\geq2$ vanishes, and the remaining terms give
\begin{equation}\label{eq:intersection-two-terms}
\bar L(\epsilon\mathbf f)^{g+1}
=(g+1)\epsilon\sum_{v\in\Bad(A/K)}
\left(
\int_{A_v^{\an}}f_v\alpha_v^g
+\frac{g\epsilon}{2}\int_{A_v^{\an}}f_v\alpha_v^{g-1}\beta_v
\right).
\end{equation}
The curvature estimate in Proposition~\ref{prop:local}(ii) is equivalent to $-\alpha_v\leq\beta_v\leq\alpha_v$ in the rank-one tropical direction. Since $f_v\geq0$,
\[
\int_{A_v^{\an}}f_v\alpha_v^{g-1}\beta_v
\geq
-\int_{A_v^{\an}}f_v\alpha_v^g.
\]
Moreover, by the definition of $I_v$,
\[
\int_{A_v^{\an}}f_v\alpha_v^g=\deg_L(A)I_v.
\]
Therefore
\begin{equation}\label{eq:positive-intersection}
\bar L(\epsilon\mathbf f)^{g+1}
\geq
\epsilon(g+1)\left(1-\frac{g\epsilon}{2}\right)
\deg_L(A)\sum_{v\in\Bad(A/K)}I_v.
\end{equation}

Apply the small-section consequence \eqref{eq:small-section-input} to the nef adelic line bundle $\bar L(\epsilon\mathbf f)$. There are integers $m_j\to\infty$ and nonzero sections $s_j\in H^0(A,m_jL)$ such that
\[
\frac1{m_j}h_{m_j\bar L(\epsilon\mathbf f)}(s_j)
\geq
\frac{\bar L(\epsilon\mathbf f)^{g+1}}
{D(g+1)\deg_L(A)}+o(1).
\]
Taking $\epsilon=1/g$ in \eqref{eq:positive-intersection} proves the following.

\begin{prop}\label{prop:big}
Fix $\epsilon=1/g$. There are integers $m_j\to\infty$ and nonzero sections $s_j\in H^0(A,m_jL)$ such that
\begin{equation}\label{eq:small-sections-final}
\frac1{m_j}h_{m_j\bar L(\epsilon\mathbf f)}(s_j)
\geq
\frac1{2gD}\sum_{v\in\Bad(A/K)}I_v+o(1)
\qquad(j\to\infty).
\end{equation}
\end{prop}
\endgroup

\section{Multiplicity estimate and the global Liouville inequality}\label{sec:Liouville}

\begingroup
We now combine Proposition~\ref{prop:big} with the multiplicity estimate and the slope bound from \cite{GaoGu_UBC}.

\begin{thm}\label{thm:local-global}
Assume $\Bad(A/K)\neq\varnothing$. For every finite subgroup $\Gamma\subset A(K)_{\tor}$ there is a proper abelian subvariety $A'\subsetneq A$, defined over $K$, such that
\begin{equation}\label{eq:local-global-minimand}
\left(
\#\left(\frac{\Gamma+A'}{A'}\right)
\frac{h^0(A',L|_{A'})}{h^0(A,L)}
\right)^{\!1/(g-\dim A')}
\leq
32g^3(g+2)\frac{h_{\mathrm{Fal}}(A)}{S(A)}.
\end{equation}
\end{thm}

Define
\begin{equation}\label{eq:def-NGamma}
N_\Gamma
:=
\min_{A'\subsetneq A}
\left(
\#\left(\frac{\Gamma+A'}{A'}\right)
\frac{h^0(A',L|_{A'})}{h^0(A,L)}
\right)^{\!1/(g-\dim A')},
\end{equation}
where $A'$ ranges over the proper $K$-abelian subvarieties of $A$ and where $h^0(0,L|_0)=1$. By \cite[Lemma~6.1]{GR25}, allowing abelian subvarieties only after base change to $\bar K$ gives the same minimum. Thus Theorem~\ref{thm:local-global} is equivalent to
\begin{equation}\label{eq:NGamma-bound}
N_\Gamma\leq32g^3(g+2)\frac{h_{\mathrm{Fal}}(A)}{S(A)}.
\end{equation}

\begin{lemma}[Multiplicity estimate]\label{lem:zero}
Let $m>N_\Gamma/g$ and let $0\neq s\in H^0(A,mL)$. Put
\[
\Gamma[0]:=\{0\},
\qquad
\Gamma[i]:=\{x_1+\cdots+x_i:x_j\in\Gamma\}\quad(i\geq1).
\]
Then there is a point
\[
x\in\bigcup_{i=0}^g\Gamma[i]
\]
such that, with $\ell=\operatorname{ord}_x(s)$,
\begin{equation}\label{eq:multiplicity}
\ell\leq g\left(\frac{gm}{N_\Gamma}+1\right).
\end{equation}
\end{lemma}

\begin{proof}
This is Nakamaye's multiplicity estimate \cite{Nakamaye}, applied with the degree convention in \eqref{eq:def-NGamma}; see \cite[Proposition~7.2]{GaoGu_UBC} for the complete reduction to that result.
\end{proof}

Let $\overline{\Lie A}$ denote the adelically normed tangent space obtained from the Hodge bundle of the N\'eron model, and endow its dual and symmetric powers with the induced quotient/tensor norms. If $x\in A(K)$ and $\ell=\operatorname{ord}_x(s)$, the {\it $\ell$-th jet} of $s$ at $x$, denoted by $\mathrm{jet}^\ell s(x)$, is the image of $s$ under the composition
\[
H^0(A,\cI_x^\ell+mL) \rightarrow H^0(\mathrm{Spec}\bar{K}, x^*\!\!\left(\frac{\cI_x^\ell}{\cI_x^{\ell+1}}+mL\right) ) \rightarrow H^0(\mathrm{Spec}\bar{K}, x^*\mathrm{Sym}^\ell(\Omega_A +mL)).
\]
Let
\[
J:=\Sym^\ell(\Lie A)^\vee\otimes L_x^{\otimes m}.
\]
Then $\mathrm{jet}^\ell s(x)$ is an element in $J$, which is non-zero  by the definition of vanishing order. 

We endow $J$ with the tensor-product adelic norm. The following result comes from the Liouville inequality; see \cite[$\mathsection$7.4]{GaoGu_UBC}.

\begin{lemma}\label{lem:liouville}
If $x\in A(K)_{\tor}$ and $\operatorname{jet}_x^\ell(s)\neq0$, then
\begin{equation}\label{eq:global-liouville}
-\frac1D\sum_{v\in M_K}
\log\|\operatorname{jet}^\ell s(x)\|_v
\leq
\ell\left(\frac g2+1\right)h_{\mathrm{Fal}}(A).
\end{equation}
\end{lemma}

\begin{proof}
The line $K\operatorname{jet}^\ell s(x)$ is an adelic subspace of $J$, hence
\begin{equation}\label{eq:jet-slope-line}
\widehat\mu\bigl(K\operatorname{jet}^\ell s(x)\bigr)
\leq\widehat\mu_{\max}(J).
\end{equation}
The left-hand side is the left-hand side of \eqref{eq:global-liouville}. Since tensoring with an adelic line translates every slope, \cite[Proposition~5.7(1)]{GaudronAdelicSpace} gives
\begin{equation}\label{eq:jet-max-slope}
\widehat\mu_{\max}(J)
=m h_{\bar L}(x)
+\widehat\mu_{\max}\bigl(\Sym^\ell(\overline{\Lie A}^{\,\vee})\bigr).
\end{equation}
The first term is zero because $x$ is torsion and $\bar L$ is canonical. In characteristic zero, \cite[Theorem~7.1(1)]{GaudronAdelicSpace} gives
\[
\widehat\mu_{\max}\bigl(\Sym^\ell(\overline{\Lie A}^{\,\vee})\bigr)
=\ell\widehat\mu_{\max}(\overline{\Lie A}^{\,\vee}),
\]
and \cite[Appendix~A]{GaoGu_UBC} proves
\[
\widehat\mu_{\max}(\overline{\Lie A}^{\,\vee})
\leq\left(\frac g2+1\right)h_{\mathrm{Fal}}(A).
\]
Combining these inequalities proves the lemma.
\end{proof}

\begin{lemma}\label{lem:local-jet}
For every $v\in M_K$,
\begin{equation}\label{eq:local-jet-bound}
\|\operatorname{jet}^\ell s(x)\|_v
\leq
\|s\|_{m\bar L(\epsilon\mathbf f),\sup,v}.
\end{equation}
\end{lemma}

\begin{proof}
The local jet estimate of \cite[Section~6.7.1]{GR_min_isog}, in the form recorded in \cite[Lemma~7.4]{GaoGu_UBC}, gives
\[
\|\operatorname{jet}^\ell s(x)\|_v
\leq\sup_{y\in A(K_v)}\|s(y)\|_{m\bar L,v}.
\]
At a good place $f_v=0$. At a bad place, the tropicalization of a $K_v$-rational point belongs to the integral subgroup $\Phi_v\subset\Sigma_v$ (see \cite[Section~4.2.3]{GaoGu_UBC}), and Proposition~\ref{prop:local}(i) gives $f_v|_{\Phi_v}=0$. Thus for every $y\in A(K_v)$,
\[
\|s(y)\|_{m\bar L,v}
=\|s(y)\|_{m\bar L(\epsilon\mathbf f),v}
\leq\|s\|_{m\bar L(\epsilon\mathbf f),\sup,v},
\]
which proves \eqref{eq:local-jet-bound}.
\end{proof}

\begin{proof}[Proof of Theorem~\ref{thm:local-global}]
Take the sequence $(m_j,s_j)$ supplied by Proposition~\ref{prop:big}. For all sufficiently large $j$, one has $m_j>N_\Gamma/g$. Choose $x_j$ and $\ell_j=\operatorname{ord}_{x_j}(s_j)$ as in Lemma~\ref{lem:zero}. Since each $x_j$ is a sum of torsion points, it is torsion.

By Lemma~\ref{lem:local-jet} and \eqref{eq:small-sections-final},
\begin{align}
-\frac1D\sum_{v\in M_K}
\log\|\operatorname{jet}^{\ell_j}s_j(x_j)\|_v
&\geq h_{m_j\bar L(\epsilon\mathbf f)}(s_j)\notag\\
&\geq\frac{m_j}{2gD}\sum_{v\in\Bad(A/K)}I_v+o(m_j).
\label{eq:jet-lower}
\end{align}
On the other hand, Lemmas~\ref{lem:liouville} and \ref{lem:zero} give
\begin{align}
-\frac1D\sum_{v\in M_K}
\log\|\operatorname{jet}^{\ell_j}{s_j}(x_j)\|_v
&\leq\ell_j\left(\frac g2+1\right)h_{\mathrm{Fal}}(A)\notag\\
&\leq g\left(\frac{gm_j}{N_\Gamma}+1\right)
\left(\frac g2+1\right)h_{\mathrm{Fal}}(A).
\label{eq:jet-upper}
\end{align}
Divide by $m_j$ and let $j\to\infty$. We obtain
\begin{equation}\label{eq:NGamma-before-local}
N_\Gamma
\leq
g^3(g+2)h_{\mathrm{Fal}}(A)
\left(\frac1D\sum_{v\in\Bad(A/K)}I_v\right)^{-1}.
\end{equation}
By Proposition~\ref{prop:local}(iii),
\[
\frac1D\sum_{v\in\Bad(A/K)}I_v
\geq
\left(\frac1{32}-\eta\right)
\frac1D\sum_{v\in\Bad(A/K)}\frac1{\#\Phi_v}
=\left(\frac1{32}-\eta\right)S(A).
\]
Letting $\eta\to0^+$ in \eqref{eq:NGamma-before-local} gives \eqref{eq:NGamma-bound}, and hence proves Theorem~\ref{thm:local-global}.
\end{proof}
\endgroup

\section{\texorpdfstring{Uniform bound over $\PP^1$}{Uniform bound over the projective line}}

\begingroup
We now derive Theorem~\ref{thm:P1} from Theorem~\ref{thm:local-global}. For clarity, we restate it.

\begin{thm}\label{thm:P1_Proof}
Let $K=\C(t)$ and let $A/K$ be a semistable abelian variety of dimension $g$, endowed with a polarization $L$. Assume that the toric rank of $A$ is at most one at every place. Let $A_0\subseteq A$ be the maximal abelian subvariety with everywhere good reduction. If $\Gamma\subset A(K)_{\tor}$ is finite and $\Gamma\cap A_0(K)=\{0\}$, then
\begin{equation}\label{eq:P1-bound-proof}
\#\Gamma\leq h^0(A,L)\bigl(768g^5(g+2)\bigr)^g.
\end{equation}
\end{thm}

Recall
\[
h_\partial(A)=\frac1D\sum_{v\in\Bad(A/K)}\#\Phi_v.
\]
The boundary estimate proved in \cite[Proposition~8.1]{GaoGu_UBC} is
\begin{equation}\label{eq:boundary}
h_\partial(A)\leq96h_{\mathrm{Fal}}(A).
\end{equation}
For a principally polarized abelian variety, $96$ may be replaced by $12$, although we do not need that improvement here.

We now assume $B=\PP^1$, so $D=1$, and put
\[
s:=\#\Bad(A/K).
\]
If $s>0$, Cauchy's inequality gives
\begin{equation}\label{eq:cauchy}
h_\partial(A)S(A)
=\left(\sum_{v\in\Bad(A/K)}\#\Phi_v\right)
\left(\sum_{v\in\Bad(A/K)}\frac1{\#\Phi_v}\right)
\geq s^2.
\end{equation}

We also use Deligne's Arakelov inequality.
\begin{thm}[Deligne]\label{thm:Deligne-Arakelov}
Let $A/B$ be a semistable abelian variety of relative dimension $g$, and let $q_B=g(B)$. Then
\begin{equation}\label{eq:arakelov}
D h_{\mathrm{Fal}}(A)
\leq
\max\left\{\frac g2(2q_B-2+s),0\right\}.
\end{equation}
Moreover, $h_{\mathrm{Fal}}(A)=0$ if and only if $A$ is isotrivial.
\end{thm}

\begin{proof}
This is the Arakelov inequality for weight-one variations of Hodge structure; see \cite{Deligne87}. The final assertion is the equality case used in the same form in \cite{GaoGu_UBC}.
\end{proof}

For $B=\PP^1$, \eqref{eq:arakelov} becomes
\begin{equation}\label{eq:arakelov-P1}
h_{\mathrm{Fal}}(A)\leq\max\left\{\frac g2(s-2),0\right\}.
\end{equation}

\begin{lemma}\label{lem:ratio}
Assume $B=\PP^1$ and $s>0$. Then
\[
\frac{h_{\mathrm{Fal}}(A)}{S(A)}\leq24g^2.
\]
\end{lemma}

\begin{proof}
Since $s>0$, one has $h_\partial(A)>0$, and \eqref{eq:boundary} implies $h_{\mathrm{Fal}}(A)>0$. Thus the positive term occurs in \eqref{eq:arakelov-P1}, so $s\geq3$ and
\[
h_{\mathrm{Fal}}(A)\leq\frac g2(s-2).
\]
Using \eqref{eq:cauchy} and \eqref{eq:boundary},
\[
S(A)\geq\frac{s^2}{h_\partial(A)}
\geq\frac{s^2}{96h_{\mathrm{Fal}}(A)}.
\]
Therefore
\[
\frac{h_{\mathrm{Fal}}(A)}{S(A)}
\leq96\left(\frac{h_{\mathrm{Fal}}(A)}s\right)^2
\leq24g^2\left(\frac{s-2}{s}\right)^2
\leq24g^2.
\]
\end{proof}

\begin{proof}[Proof of Theorem~\ref{thm:P1_Proof}]
For $d\geq1$, put
\[
C_d:=768d^5(d+2),
\qquad C_0:=1.
\]
We prove by induction on $g=\dim A$ that
\begin{equation}\label{eq:induction-goal}
\#\Gamma\leq h^0(A,L)C_g^g.
\end{equation}
If $A$ has good reduction everywhere, then $A=A_0$ and the hypothesis gives $\Gamma=\{0\}$, so the assertion is immediate.

Assume now that $\Bad(A/K)\neq\varnothing$. By Theorem~\ref{thm:local-global} and Lemma~\ref{lem:ratio}, there is a proper abelian subvariety $A'\subsetneq A$, of dimension $h$, such that
\begin{equation}\label{eq:induction-quotient}
\#\left(\frac{\Gamma+A'}{A'}\right)
\leq
\frac{h^0(A,L)}{h^0(A',L|_{A'})}C_g^{g-h}.
\end{equation}
Put $\Gamma':=\Gamma\cap A'(K)$. Semistability and the condition on toric ranks are inherited by $A'$. Moreover, the maximal everywhere-good-reduction abelian subvariety $A'_0$ of $A'$ is contained in $A_0$. Hence
\[
\Gamma'\cap A'_0(K)=\{0\}.
\]
The induction hypothesis, with the evident convention when $h=0$, gives
\[
\#\Gamma'\leq h^0(A',L|_{A'})C_h^h.
\]
Since
\[
\#\Gamma
=\#\Gamma'\cdot\#\left(\frac{\Gamma+A'}{A'}\right),
\]
we obtain from \eqref{eq:induction-quotient}
\[
\#\Gamma
\leq h^0(A,L)C_h^hC_g^{g-h}
\leq h^0(A,L)C_g^g,
\]
because $C_h\leq C_g$. This completes the induction.
\end{proof}
\endgroup

\section{\texorpdfstring{Lefschetz fibrations and the bound for $\NS(X)_{\tor}$}{Lefschetz fibrations and the Neron--Severi torsion bound}}

\begingroup
\subsection{Reduction to surfaces}

\begin{lemma}\label{lem:weak-lefschetz-torsion}
Let $X$ be a smooth projective complex variety of dimension $n\geq2$, let $L$ be very ample, and let
\[
X=X_n\supset X_{n-1}\supset\cdots\supset X_2=S
\]
be a smooth flag such that $X_{i-1}$ is a general member of $|L|_{X_i}|$. Then restriction induces an injection
\[
\NS(X)\hookrightarrow\NS(S).
\]
\end{lemma}

\begin{proof}
Integral weak Lefschetz gives an isomorphism
\[
H^2(X_i,\Z)\longrightarrow H^2(X_{i-1},\Z)
\]
for $i\geq4$, and an injection for $i=3$. Hence $H^2(X,\Z)\to H^2(S,\Z)$ is injective. For every smooth projective complex variety $Z$, the exponential sequence identifies the kernel of $c_1:\Pic(Z)\to H^2(Z,\Z)$ with $\Pic^0(Z)$; consequently
\[
c_1:\NS(Z)=\Pic(Z)/\Pic^0(Z)\hookrightarrow H^2(Z,\Z)
\]
is injective. If a class in $\NS(X)$ restricts to zero in $\NS(S)$, its Chern class restricts to zero in $H^2(S,\Z)$, hence is zero already on $X$. The class itself is therefore zero.
\end{proof}

\begin{thm}\label{thm:surface}
Let $X$ be a smooth projective connected complex surface and let $L$ be very ample. Put
\[
g(L):=1+\frac12(L^2+K_X\cdot L),
\]
the genus of a smooth member of $|L|$. Then
\[
\#\NS(X)_{\tor}\leq\Bfun(g(L)).
\]
\end{thm}

Theorem~\ref{thm:surface} implies Theorem~\ref{thm:polarized-variety}. Indeed, choose a flag as in Lemma~\ref{lem:weak-lefschetz-torsion}. The surface $S=X_2$ is the complete intersection of $n-2$ divisors linearly equivalent to $L$, so adjunction gives
\[
K_S=(K_X+(n-2)L)|_S.
\]
For a smooth $C\in|L|_S|$, a second application of adjunction yields
\[
2g(C)-2=(K_S+C)\cdot C
=(K_X+(n-1)L)\cdot L^{n-1}.
\]
Thus $g(C)=q(X,L)$, and Lemma~\ref{lem:weak-lefschetz-torsion} together with Theorem~\ref{thm:surface} gives Theorem~\ref{thm:polarized-variety}.

\subsection{Construction of a Lefschetz fibration}
A point $x$ on a variety over an algebraically closed field is an \emph{ordinary double point} if its completed local ring is
\[
\widehat\cO_{X,x}\simeq k[[x_1,\ldots,x_n]]/(f),
\]
where $f\in\mathfrak m^2$ and the image of $f$ in $\mathfrak m^2/\mathfrak m^3$ is a nondegenerate quadratic form.

For a smooth projective complex surface $X$ and a very ample line bundle $L$, a \emph{Lefschetz pencil in $|L|$} is a pencil whose base locus is reduced and whose singular members are reduced curves with exactly one ordinary double point. Blowing up the base locus turns the pencil into a proper flat morphism to $\PP^1$.

\begin{thm}[Lefschetz pencil theorem]\label{thm:LefFibExist}
Let $X$ be a smooth projective complex surface and let $L$ be very ample. For a general two-dimensional subspace $\langle s_0,s_1\rangle\subset H^0(X,L)$, the base locus
\[
Z=V(s_0)\cap V(s_1)
\]
consists of $L^2$ distinct points, the critical points of the pencil lie outside $Z$, and
\[
p:\widetilde X:=\operatorname{Bl}_ZX\longrightarrow\PP^1,
\qquad
x\longmapsto[s_0(x):s_1(x)],
\]
is a semistable fibration whose singular fibers are reduced and have exactly one ordinary double point.
\end{thm}

\begin{proof}
This is the classical Lefschetz pencil theorem in characteristic zero; see \cite[Chapter~III, Proposition~1.10]{FreitagKiehl_EtCohandWeilConj}. The transversality assertions follow by choosing the line in the dual projective space outside the standard proper discriminant loci.
\end{proof}

For $z\in Z$, denote by $E_z$ the exceptional curve over $z$. Since $\#Z=L^2>0$, at least one such curve exists.

\begin{lemma}\label{lem:ExcepDivSection}
Every exceptional curve $E_z$ is a section of $p$.
\end{lemma}

\begin{proof}
Near $z$, trivialize $L$ and choose regular parameters $u,v$ such that $s_0=u$ and $s_1=v$. The blowup is locally
\[
\operatorname{Bl}_{(u,v)}\A^2
=\{((u,v),[a:b]):ub=va\},
\]
and the induced morphism sends $((u,v),[a:b])$ to $[a:b]$. On the exceptional curve $u=v=0$, this is the identity map $\PP^1_{[a:b]}\to\PP^1$. Thus $E_z$ is a section.
\end{proof}

Fix from now on a Lefschetz fibration $p:\widetilde X\to\PP^1$ as above. Its generic fiber is a smooth connected curve of genus $g(L)$. Put $K=\C(\PP^1)$ and let
\[
A:=\Pic^0(\widetilde X_\eta)
\]
be its Jacobian. We use $\Pic^0_{\widetilde X/\PP^1}$ for the multidegree-zero identity component of the relative Picard functor.

\begin{lemma}\label{lem:semiabelian-picard}
The group scheme $\Pic^0_{\widetilde X/\PP^1}$ is the identity component of the N\'eron model of $A$. It is a semiabelian scheme, and its toric rank at every point of $\PP^1$ is at most one.
\end{lemma}

\begin{proof}
Because $p$ is a regular semistable curve, \cite[Theorem~9.5/4]{BLR} identifies the multidegree-zero relative Picard scheme with the identity component of the N\'eron model of the generic Jacobian.

Let $C$ be a geometric fiber, let $\widetilde C_i$ be the normalizations of its irreducible components, and let $\Gamma_C$ be its dual graph. The generalized Jacobian fits into the standard exact sequence
\begin{equation}\label{eq:generalized-jacobian}
0\longrightarrow
\Hom\bigl(H_1(\Gamma_C,\Z),\Gm\bigr)
\longrightarrow\Pic^0(C)
\longrightarrow\prod_i\Pic^0(\widetilde C_i)
\longrightarrow0.
\end{equation}
Thus the toric rank is $b_1(\Gamma_C)$. A fiber of a Lefschetz fibration is smooth or has one node. In the latter case, the dual graph has first Betti number $0$ if the fiber is reducible and $1$ if it is irreducible. Hence the toric rank is at most one.
\end{proof}

Since $\widetilde X$ is the blowup of the smooth surface $X$ at the reduced set $Z$, pullback and the exceptional classes give
\begin{equation}\label{eq:blowup-picard}
\Pic(\widetilde X)
\simeq
\Pic(X)\oplus\bigoplus_{z\in Z}\Z[E_z],
\qquad
\NS(\widetilde X)
\simeq
\NS(X)\oplus\bigoplus_{z\in Z}\Z[E_z].
\end{equation}
In particular,
\begin{equation}\label{eq:blowup-torsion}
\NS(\widetilde X)_{\tor}\simeq\NS(X)_{\tor}.
\end{equation}

Let $(T_0,\tau)$ be the $K/\C$-trace of $A$, and let
\[
T:=\tau((T_0)_K)\subseteq A
\]
be its image. In characteristic zero the finite kernel of $\tau$ is constant, so the trace image $T$ is a constant $K$-abelian subvariety. Let
\[
B_0:=\operatorname{Im}\bigl(\Pic^0(\widetilde X)_K\longrightarrow A\bigr).
\]
By the universal property of the trace, $B_0\subseteq T$.

\begin{lemma}\label{lem:Fixedpart}
One has $B_0=T$.
\end{lemma}

\begin{proof}
Let $U\subset\PP^1$ be the smooth locus of $p$ and let $F$ be a smooth geometric fiber. Deligne's global invariant-cycle theorem and fixed-part theorem \cite[Theorem~4.1.1 and Corollary~4.1.2]{DeligneHodgeII} identify the image of
\[
H^1(\widetilde X,\Q)\longrightarrow H^1(F,\Q)
\]
with the monodromy-invariant Hodge substructure
\[
H^0(U,R^1p_*\Q).
\]
For the family of Jacobians, this invariant Hodge structure is precisely $H^1(T,\Q)$. On the other hand, the image of $H^1(\widetilde X,\Q)$ in $H^1(F,\Q)$ is the rational Hodge structure attached to the image abelian subvariety $B_0$. Hence $B_0$ and $T$ have the same dimension. Since $B_0\subseteq T$, equality follows.
\end{proof}

\begin{lemma}\label{lem:NS-to-Jacobian}
Restriction to the generic fiber defines an injection
\[
\rho:\NS(\widetilde X)_{\tor}
\hookrightarrow
(A/T)(K)_{\tor}.
\]
\end{lemma}

\begin{proof}
Let $[\cM]\in\NS(\widetilde X)_{\tor}$. Since intersection with a fiber is a homomorphism $\NS(\widetilde X)\to\Z$, torsion implies
\[
\deg(\cM|_{\widetilde X_\eta})=\cM\cdot F=0.
\]
Thus $\cM|_{\widetilde X_\eta}$ defines a point of $A(K)$. Replacing $\cM$ by an algebraically equivalent line bundle changes this point by the restriction of an element of $\Pic^0(\widetilde X)$, whose image is $T$ by Lemma~\ref{lem:Fixedpart}. Hence the class in $(A/T)(K)$ is well defined. If $n[\cM]=0$ in $\NS(\widetilde X)$, then $\cM^{\otimes n}\in\Pic^0(\widetilde X)$, so $\rho([\cM])$ is torsion.

It remains to prove injectivity. Suppose $\rho([\cM])=0$. Then $\cM|_{\widetilde X_\eta}$ is a $K$-point of $T$. Since $T$ is constant and $K=\C(\PP^1)$, every $K$-point of $T$ is constant: a rational map $\PP^1\dashrightarrow T$ extends to a morphism and every morphism from $\PP^1$ to an abelian variety is constant. Thus
\[
T(K)=T(\C).
\]
By Lemma~\ref{lem:Fixedpart}, the surjective morphism $\Pic^0(\widetilde X)\to T$ is also surjective on complex points. After twisting $\cM$ by an element of $\Pic^0(\widetilde X)$, without changing its N\'eron--Severi class, we may therefore assume that $\cM|_{\widetilde X_\eta}$ is trivial. A rational trivializing section then represents $\cM$ by a vertical Cartier divisor, because $\widetilde X$ is regular.

Let $V\subset\NS(\widetilde X)$ be the subgroup generated by vertical divisor classes. Choose one exceptional section $E$. In each reducible fiber, omit the unique irreducible component met by $E$ and retain all the other components. The class $F$ of a fiber, together with all retained components, generates $V$: the omitted component equals $F$ minus the sum of the retained components in that fiber.

Suppose that a rational linear combination of these generators is zero. Intersecting with $E$ first shows that the coefficient of $F$ is zero, because $E\cdot F=1$ and $E$ is disjoint from every retained component. For each reducible fiber separately, the intersection matrix of the retained components is negative definite by Zariski's lemma, so all their coefficients vanish. The generators are therefore $\Q$-linearly independent, and $V$ is a free abelian group. Since $[\cM]$ is torsion and lies in $V$, it is zero. Thus $\rho$ is injective.
\end{proof}

\begin{lemma}\label{lem:trace-good}
The trace image $T$ is the maximal abelian subvariety $A_0\subseteq A$ having everywhere good reduction.
\end{lemma}

\begin{proof}
The constant abelian subvariety $T$ has good reduction everywhere, so $T\subseteq A_0$. Conversely, let $C\subseteq A$ be an abelian subvariety with good reduction everywhere. Its N\'eron model is an abelian scheme over $\PP^1$. Applying Theorem~\ref{thm:Deligne-Arakelov} with $s=0$ gives $h_{\mathrm{Fal}}(C)=0$, hence $C$ is isotrivial. After imposing a full level structure of order at least three, an isotrivial abelian scheme becomes constant after a finite \emph{\'etale} cover of the base. Since $\PP^1_\C$ has no nontrivial connected finite \'etale cover, $C$ is already constant. The universal property of the trace then gives $C\subseteq T$. Hence $A_0=T$.
\end{proof}

\begin{proof}[Proof of Theorem~\ref{thm:surface}]
If $g(L)=0$, then the generic fiber is a rational curve and $A=0$. Lemma~\ref{lem:NS-to-Jacobian} therefore gives $\NS(\widetilde X)_{\tor}=0$, and the result follows from \eqref{eq:blowup-torsion} and $\Bfun(0)=1$.

Assume $g(L)\geq1$. Let
\[
\Gamma':=\operatorname{Im}(\rho)\subseteq(A/T)(K)_{\tor}.
\]
This is finite. Every element of $\Gamma'$ lies in the image of $A(K)$ by construction. Let $E_\Gamma\subseteq A(K)$ be its full inverse image. Then
\[
0\longrightarrow T(K)=T(\C)
\longrightarrow E_\Gamma
\longrightarrow\Gamma'
\longrightarrow0
\]
is exact. The group $T(\C)$ is divisible and therefore injective in the category of abelian groups; the sequence splits. Consequently there is a finite subgroup $\Gamma\subset A(K)_{\tor}$ that maps isomorphically onto $\Gamma'$ and satisfies
\[
\Gamma\cap T(K)=\{0\}.
\]
By Lemma~\ref{lem:trace-good}, $T=A_0$.

The Jacobian $A$ has dimension $g(L)$ and carries its principal theta polarization $\Theta$, for which $h^0(A,\Theta)=1$. Lemma~\ref{lem:semiabelian-picard} verifies semistability and the rank-one toric condition. Theorem~\ref{thm:P1} therefore gives
\[
\#\Gamma
\leq
\bigl(768g(L)^5(g(L)+2)\bigr)^{g(L)}
=\Bfun(g(L)).
\]
Finally, Lemma~\ref{lem:NS-to-Jacobian} and \eqref{eq:blowup-torsion} give
\[
\#\NS(X)_{\tor}
=\#\NS(\widetilde X)_{\tor}
=\#\Gamma'
=\#\Gamma
\leq\Bfun(g(L)).
\]
\end{proof}
\endgroup

\section{Bound for torsion sections of an abelian scheme}

\begingroup
We now prove Corollary~\ref{cor:AV}. Let
\[
\pi:\calA\longrightarrow B
\]
be a traceless abelian scheme over a smooth projective connected complex curve of genus $q=g(B)$, and let $g=\dim(\calA/B)\geq1$. Fix a relatively very ample symmetric line bundle $\calL$ on $\calA$, rigidified along the zero section. Write $A/K$ and $L$ for the generic fibers, where $K=\C(B)$, and put
\[
\omega_{\calA}:=\det e^*\Omega^1_{\calA/B}.
\]
If $q=0$, the Arakelov inequality with no bad fibers gives $\deg\omega_{\calA}=0$, so the abelian scheme is isotrivial; because $\PP^1$ has no nontrivial finite \'etale covers, it is constant. Tracelessness would then force $g=0$. Thus, in the nontrivial case considered here, $q\geq1$.

\begin{thm}\label{thm:NSboundA}
Define
\[
Q:=\frac12(g+3)!\,h^0(A,L)(q+1).
\]
Then
\[
\#\NS(\calA)_{\tor}\leq\Bfun(Q).
\]
\end{thm}

Let
\[
E:=\pi_*\calL,
\qquad
r:=\rank E=h^0(A,L).
\]
Since $\calL$ is relatively ample on an abelian scheme, the higher cohomology of its restrictions to the fibers vanishes; cohomology and base change therefore show that $E$ is locally free and commutes with base change.

\begin{lemma}[Szpiro--Moret--Bailly]\label{lem:MB-formula}
For a rigidified relatively ample symmetric line bundle $\calL$ on an abelian scheme $\calA\to B$,
\begin{equation}\label{eq:MB-formula}
\frac{\deg E}{\rank E}=-\frac12\deg\omega_{\calA}.
\end{equation}
\end{lemma}

\begin{proof}
This is Moret--Bailly's formula; see \cite[Chapter~VIII, Theorem~1.3]{MB_model}.
\end{proof}

\begin{lemma}\label{lem:E-semistable}
The vector bundle $E=\pi_*\calL$ is semistable.
\end{lemma}

\begin{proof}
The relative theta group $K(\calL)$ acts on $E$, and its generic-fiber representation on $H^0(A,L)$ is the irreducible weight-one theta representation. If $E$ were not semistable, its maximal destabilizing subbundle would be unique and therefore invariant under every automorphism coming from $K(\calL)$. Its generic fiber would then be a nonzero proper invariant subspace of $H^0(A,L)$, contradicting irreducibility. This is the function-field form of the argument in \cite[Proposition~A.3]{Bost_MWBourbaki}.
\end{proof}

Combining Lemmas~\ref{lem:MB-formula} and \ref{lem:E-semistable}, we have
\begin{equation}\label{eq:E-slope}
\mu(E)=-\frac12\deg\omega_{\calA}.
\end{equation}

\begin{prop}\label{prop:m0}
Let $M$ be a degree-one line bundle on $B$. Then
\[
\calL\otimes\pi^*M^{\otimes m}
\]
is very ample on $\calA$ for every
\begin{equation}\label{eq:m0}
m\geq m_0:=\left\lceil2q+\frac g2(q-1)+1\right\rceil.
\end{equation}
\end{prop}

\begin{proof}
We use the quotient convention for projective bundles. Relative very ampleness gives a closed immersion
\[
\phi_{\calL}:\calA\hookrightarrow\PP_B(E)
\]
with
\[
\calL\simeq\phi_{\calL}^*\cO_{\PP(E)}(1).
\]
There is a canonical identification
\[
\PP_B(E\otimes M^{\otimes m})\simeq\PP_B(E),
\]
under which the tautological line bundle corresponds to
\[
\cO_{\PP(E)}(1)\otimes p^*M^{\otimes m}.
\]
Thus, if $E\otimes M^{\otimes m}$ is a very ample vector bundle on $B$, then $\calL\otimes\pi^*M^{\otimes m}$ is the restriction of a very ample line bundle on $\PP_B(E)$ and is therefore very ample on $\calA$.

The bundle $E\otimes M^{\otimes m}$ is semistable, and
\[
\deg(E\otimes M^{\otimes m})
=r\left(m-\frac12\deg\omega_{\calA}\right)
\]
by \eqref{eq:E-slope}. Ionescu--Toma's criterion \cite[Corollary~2]{IonescuToma_VeryAmpleVB} says that a semistable rank-$r$ vector bundle of degree at least $2qr+1$ is very ample. It is therefore enough that
\[
m\geq2q+\frac12\deg\omega_{\calA}+1.
\]
Since $\calA/B$ has good reduction everywhere, Deligne's Arakelov inequality gives
\begin{equation}\label{eq:omega-bound}
\deg\omega_{\calA}\leq g(q-1).
\end{equation}
The choice \eqref{eq:m0} satisfies the required inequality.
\end{proof}

\begin{proof}[Proof of Theorem~\ref{thm:NSboundA}]
Fix a degree-one line bundle $M$ on $B$ and set
\[
\calL':=\calL\otimes\pi^*M^{\otimes m_0}.
\]
By Proposition~\ref{prop:m0}, $\calL'$ is very ample on the smooth projective variety $\calA$ of dimension $g+1$.

The symmetry and rigidification of $\calL$ give $[n]^*\calL\simeq\calL^{\otimes n^2}$. Since $[n]:\calA\to\calA$ has degree $n^{2g}$, the projection formula yields
\[
n^{2g}\calL^{g+1}=n^{2g+2}\calL^{g+1},
\]
so
\begin{equation}\label{eq:L-top-zero}
\calL^{g+1}=0.
\end{equation}
Also $(\pi^*M)^2=0$. Since $L^g=g!h^0(A,L)$ on the generic abelian variety,
\begin{align}
(\calL')^{g+1}
&=(g+1)m_0\,\pi^*M\cdot\calL^g\notag\\
&=m_0(g+1)!h^0(A,L).
\label{eq:Lprime-top}
\end{align}

Translation-invariant differentials identify
\[
\Omega^1_{\calA/B}\simeq\pi^*e^*\Omega^1_{\calA/B},
\]
so
\[
K_{\calA/B}=\pi^*\omega_{\calA},
\qquad
K_{\calA}=\pi^*(K_B\otimes\omega_{\calA}).
\]
Because the square of every divisor pulled back from $B$ is zero,
\begin{align}
K_{\calA}\cdot(\calL')^g
&=\pi^*(K_B\otimes\omega_{\calA})\cdot\calL^g\notag\\
&=(2q-2+\deg\omega_{\calA})g!h^0(A,L)\notag\\
&\leq(g+2)(q-1)g!h^0(A,L),
\label{eq:K-intersection}
\end{align}
where the last inequality is \eqref{eq:omega-bound}.

The integer in Theorem~\ref{thm:polarized-variety} for $(\calA,\calL')$ is
\[
q(\calA,\calL')
=1+\frac12\left(K_{\calA}\cdot(\calL')^g
+g(\calL')^{g+1}\right).
\]
Since
\begin{equation}\label{eq:m0-upper}
m_0
=2q+1+\left\lceil\frac{g(q-1)}2\right\rceil
\leq\frac{g+4}{2}q+1,
\end{equation}
Equations~\eqref{eq:Lprime-top} and \eqref{eq:K-intersection} give
\begin{align*}
q(\calA,\calL')
&\leq
1+\frac{g!h^0(A,L)}2
\bigl((g+2)(q-1)+m_0g(g+1)\bigr)\\
&\leq
\frac{g!h^0(A,L)}2
\left((g+2)q+\left(\frac{g+4}{2}q+1\right)g(g+1)\right)\\
&<
\frac{g!h^0(A,L)}2
\left((g+3)(g+2)(g+1)q+g(g+1)\right)\\
&<
\frac12(g+3)!h^0(A,L)(q+1)=Q.
\end{align*}
Here the second line absorbs the initial $1$ because $g!h^0(A,L)(g+2)/2\geq1$.

The function $\Bfun(t)$ is increasing on the nonnegative integers. Applying Theorem~\ref{thm:polarized-variety} to $(\calA,\calL')$ therefore gives
\[
\#\NS(\calA)_{\tor}
\leq\Bfun\bigl(q(\calA,\calL')\bigr)
\leq\Bfun(Q).
\]
\end{proof}

\begin{proof}[Proof of Corollary~\ref{cor:AV}]
Set
\[
\cZ:=(\calA\times_B\calA^\vee)^4.
\]
Then $\cZ/B$ is traceless of relative dimension $8g$, and the inclusion into the first factor gives
\[
\#\calA(B)_{\tor}\leq\#\cZ(B)_{\tor}.
\]
By \cite[Section~11.1.2]{GaoGu_UBC}, torsion sections of a traceless abelian scheme are identified with the N\'eron--Severi torsion of its dual:
\[
\cZ(B)_{\tor}\simeq\NS(\cZ^\vee)_{\tor}.
\]
Zarhin's trick gives a principal polarization
\[
\lambda:\cZ\xrightarrow{\sim}\cZ^\vee
\]
represented by a rigidified relatively ample symmetric line bundle $\calL_\lambda$; see \cite[Appendix~B]{GaoGu_UBC}. The tensor cube $\calL_\lambda^{\otimes3}$ is relatively very ample, and on the generic fiber
\[
h^0(Z,L_\lambda^{\otimes3})=3^{8g}h^0(Z,L_\lambda)=3^{8g}.
\]
Applying Theorem~\ref{thm:NSboundA} to $(\cZ,\calL_\lambda^{\otimes3})$ gives the parameter
\[
Q=\frac12(8g+3)!\,3^{8g}(q+1).
\]
Consequently,
\[
\#\calA(B)_{\tor}
\leq\#\cZ(B)_{\tor}
=\#\NS(\cZ)_{\tor}
\leq\Bfun(Q),
\]
as claimed.
\end{proof}
\endgroup

\bibliographystyle{alpha}
\bibliography{mybib}
\end{document}